\documentclass[reqno,11pt]{amsart}

\usepackage{color}
\usepackage{graphicx, amsmath, amssymb, amscd, amsthm, euscript, psfrag, amsfonts,bm}

\usepackage{hyperref}

\usepackage[latin1]{inputenc}
\DeclareMathAlphabet{\mathpzc}{OT1}{pzc}{m}{it}
\numberwithin{equation}{section}
\theoremstyle{plain}
\newtheorem*{maintheorem*}{Main Theorem}
\newtheorem*{thm*}{Theorem}
\newtheorem*{thma*}{Theorem A}
\newtheorem*{thmaa*}{Theorem A'}
\newtheorem*{thmb*}{Theorem B}
\newtheorem*{thmo*}{Theorem 1.1}
\newtheorem*{thmc*}{Theorem C}
\newtheorem*{thmd*}{Theorem D}
\newtheorem*{thmf*}{Theorem 4.1}
\newtheorem*{remark*}{Remark}

\newtheorem*{conjecture*}{Conjecture}
\newtheorem*{prop*}{Proposition}
\newtheorem*{lem*}{Basic Lemma}
\newtheorem{thm}{Theorem}[section]

\newtheorem{lem}[thm]{Lemma}

\newtheorem{prop}[thm]{Proposition}

\theoremstyle{definition}

\newtheorem*{proofc*}{Proof of Theorem C}

\def\bbr{\mathbb{R}}

\def\tbf{{\mathbf{t}}}

\def\SL{{\rm{SL}}}
\def\PSL{{\rm PSL}}

\def\supp{{\rm supp}}

\def\t{{\tbf}}
\newcommand{\PS}{\operatorname{PS}}
\newcommand{\BR}{\operatorname{BR}}

\newcommand{\BMS}{\operatorname{BMS}}
\newcommand{\Haar}{\operatorname{Haar}}

\newcommand{\be}{\begin{equation}}
\newcommand{\ee}{\end{equation}}

\def\G{\Gamma}

\def\ba{\backslash}

\def\ep{\epsilon}

\def\br{\mathbb{R}}
\def\field{\mathbb{F}}

\newcommand{\bH}{\mathbb H}

\renewcommand{\c}{\mathbb{C}}
\newcommand{\op}{\operatorname}
\newcommand{\brg}{m^{\BR}_{\Gamma_2}}

\newcommand{{\cont}}{{\check U}}

\newcommand\gf{\PSL_2(\field)}

\newcommand\dN{\Delta(N)}
\begin{document}
\title[]
{ Invariant Radon measures for Unipotent flows and products of  Kleinian groups}

\author{Amir Mohammadi}
\address{Department of Mathematics, The University of Texas at Austin, Austin, TX 78750}
\email{amir@math.utexas.edu}

\thanks{Mohammadi was supported in part by NSF Grants \#1200388, \#1500677 and \#1128155, and Alfred P.~Sloan Research Fellowship.}

\author{Hee Oh}
\address{Mathematics department, Yale university, New Haven, CT 06511
and Korea Institute for Advanced Study, Seoul, Korea}
\email{hee.oh@yale.edu}

\thanks{Oh was supported by in parts by NSF Grant \#1361673.}






\begin{abstract} Let $G= \gf$ where $\field= \br , \c$, and
 consider the space $Z=(\Gamma_1 \times \Gamma_2)\ba (G\times G)$ where $\Gamma_1<G$ is  a co-compact lattice and $\G_2<G$ is a finitely generated discrete 
  Zariski dense subgroup. The work of Benoist-Quint  \cite{BQ} gives a classification of 
all ergodic  invariant Radon measures on $Z$ for the diagonal $G$-action. In this paper,
   for a horospherical subgroup $N$ of $G$, we
 classify all ergodic, conservative,  invariant Radon measures on $Z$
  for the diagonal $N$-action, under the additional assumption that $\Gamma_2$ is geometrically finite. 
  \end{abstract}

\maketitle

\section{Introduction}

The celebrated theorem of M. Ratner in 1992 classifies all {\it finite}  invariant measures for
  unipotent flows on the quotient space of a connected Lie group by its discrete subgroup \cite{Ra}.
  The problem of classifying invariant {\it locally finite}  Borel measures (i.e., Radon measures)
is far from being understood in general. Most of known classification results are restricted
to the class of horopsperical invariant measures on a quotient of a simple Lie group of rank one
(\cite{Bu,Roblin, Wi}, \cite{Bab,Led,LedSa,Sar}). In this article, we obtain a classification of Radon measures
invariant under unipotent flow in one of the most basic examples
of the quotient of  a higher rank semisimple Lie group by a discrete subgroup of infinite co-volume.
\medskip


Let $G= \gf$ where $\field$ is either $\br $ or $ \c$. Let $\G_1$ and $\G_2$ be
finitely generated Zariski dense, discrete subgroups of $G$. Set
$$Z:=(\Gamma_1\times \Gamma_2) \ba (G\times G)=X_1\times X_2$$
where  $X_i=\Gamma_i\ba G$ for $i=1,2$.
For $S\subset G$, $\Delta(S):=\{(s,s):s\in S\}$ denotes the diagonal embedding of $S$ into $G\times G$.

\begin{thm}[\cite{BQ}, Benoist-Quint] Assume that $\G_1<G$ is co-compact. Then
 any ergodic $\Delta(G)$-invariant Radon measure $\mu$
on $Z$ is, up to a constant multiple, one of the following:
\begin{itemize}
\item  $\mu$ is the product $m^{\Haar}\times m^{\Haar}$ of Haar measures;
  \item $\mu$ is the graph of the Haar measure, 
  in the sense that
   for some $g_0\in G$ with $[\Gamma_2: g_0^{-1}\Gamma_1 g_0\cap \Gamma_2]<\infty$,
$\mu=\iota_*m^{\Haar}_{(g_0^{-1}\Gamma_1 g_0\cap \Gamma_2)}$, 
i.e., the push-forward of the $\Haar$-measure on $(g_0^{-1}\Gamma_1 g_0\cap \Gamma_2)\ba G$
to the closed orbit $[(g_0, e)]\Delta (G)$ via the isomorphism $\iota$ given by $[g]\mapsto [(g_0g, g)]$.
\end{itemize}
\end{thm}
Indeed,  it is proved in \cite{BQ} that any ergodic $\G_2$-invariant Borel probability
measure on $X_1$ is either a Haar measure or supported on a finite orbit of $\G_2$. This result is equivalent to the above theorem,
 in view of the homeomorphism $\nu\mapsto \tilde \nu$
  between the space of all $\G_2$-invariant measures on $X_1$ and the space of all
 $\Delta(G)$-invariant measures on $Z$, given by
  $$\tilde \nu (f)=\int_{X_2}\int_{X_1} f(\Gamma_1 hg, \Gamma_2 g) \operatorname{d}\!\nu(h)\operatorname{d}\!m^{\Haar}(g). $$
Since the Haar measure $m^{\Haar}$ on $X_1$ is ergodic for any element of $G$ which generates an unbounded
subgroup, it follows that $m^{\Haar}$ is $\G_2$-ergodic and hence
the product $m^{\Haar}\times m^{\Haar}$  of the Haar measures in $X_1\times X_2$ is $\Delta(G)$-ergodic.

We now consider the action of $\Delta(N)$ on $Z$ where $N$ is a horospherical  subgroup of $G$, i.e., $N$ is conjugate to the subgroup  
$$\left\{\begin{pmatrix} 1 & 0\\t & 1\end{pmatrix}: t\in \field\right\}.$$

A $\dN$-invariant Radon measure on $Z$ is said to be {\it conservative}
if for any subset $S$ of positive measure in $Z$, the measure of $\{n\in N:
xn\in S\}$, with respect to the Haar measure of $N$, is infinite for almost all $x \in S$.

\medskip

The aim of this paper is to classify all $\dN$-invariant ergodic conservative Radon measures on $Z$ assuming $\G_2$ is geometrically {finite}.
Since Ratner~\cite{Ra} classified all  such finite measures, our focus lies on {\it infinite} Radon measures.


Note that if $\mu$ is a $\dN$-invariant measure, then the translate {$w_* \mu$}
is also $\dN$-invariant  for any $w$ in the centralizer of $\dN$. 
The centralizer of $\dN$ in {$G\times G$} is equal to $N\times N$. Hence it suffices to classify
$\dN$-invariant measures, up to a translation by an element of $N\times N$.

Let  $m^{\BR}_{\Gamma_2}$  denote the $N$-invariant Burger-Roblin measure on $X_2$.
It is known that
$m^{\BR}_{\Gamma_2}$ is the unique $N$-invariant ergodic conservative measure on $X_2$, 
which is not supported on a closed $N$-orbit (\cite{Bu}, \cite{Roblin}, \cite{Wi}).
When $\G_2$ is of infinite co-volume, $\brg$ is an infinite measure.

In the following two theorems, which are main results of this paper,
 we assume that $\Gamma_1<G$ is cocompact  and
$\G_2$ is a Zariski dense, geometrically finite subgroup of $G$ with infinite co-volume. 
\begin{thm}\label{ergg}
The product measure $m^{\Haar}\times m^{\BR}_{\Gamma_2}$ on $Z$  is a
 $\dN$-ergodic conservative infinite Radon measure.
\end{thm}

\begin{thm}\label{main}   Any
 $\dN$-invariant, ergodic,  conservative, infinite Radon measure $\mu$ on $Z$ 
 is one of the following, up to a translation by an element of $N\times N$:
\begin{enumerate}
\item $\mu$ is the product measure $m^{\Haar}\times m^{\BR}_{\Gamma_2};$
  \item $\mu$ is the graph of the $\BR$-measure, in the sense that
   for some $g_0\in \gf$ with $[\Gamma_2: g_0^{-1}\Gamma_1 g_0\cap \Gamma_2]<\infty$,
\[
 \mu=\phi_*m^{\BR}_{(g_0^{-1}\Gamma_1 g_0\cap \Gamma_2)},
\]
i.e., the push-forward of the $\BR$-measure on $(g_0^{-1}\Gamma_1 g_0\cap \Gamma_2)\ba G$
to the closed orbit $[(g_0, e)]\Delta (G)$ via the isomorphism $\phi$ given by $[g]\mapsto [(g_0g, g)]$.

\item $\field=\c$ and there exists a closed  orbit $x_2N$ in $ X_2$ homeomorphic to $\br \times \mathbb S^1$
such that $\mu$ is supported on $X_1\times x_2N$. To describe $\mu$ more precisely,
let $U< N$ denote the one dimensional subgroup containing $\op{Stab}_{N}(x_2)$
and $\operatorname{d}\!n$ the $N$-invariant measure on $x_2N$ in $X_2$.
We then have one of the two possibilities:
\begin{enumerate} \item
 $$\mu= m^{\Haar}\times \operatorname{d}\!n;$$ 
\item  there exist a 
connected subgroup $L\simeq\SL_2(\bbr)$ with $L\cap N=U$,
 a compact $L$ orbit $Y$ in $ X_1$ and an element $n\in N$ 
such  that
$$\mu=\int_{x_2N} \mu_{x}\operatorname{d}\!x$$ where $\mu_{x_2n_0}$ is given by
$\mu_{x_2n_0}(\psi )=\int_{Y} \psi (yn n_0, x_2n_0)\operatorname{d}\!y $ with
$\operatorname{d}\!y$ being the $L$-invariant probability measure on $Y$.
\end{enumerate}\end{enumerate}
 \end{thm}

We deduce Theorem \ref{ergg} as a consequence of Theorem \ref{main} (see subsection \ref{dedd}).

 Two main ingredients of the proof of Theorem \ref{main} are Ratner's { classification of probability measures on $X_1$ which are invariant and ergodic under a one parameter unipotent subgroup of $G$,} and the classification of $N$-equivariant {(set-valued) Borel maps} $X_2\to X_1$, established in our earlier work
 \cite{MO}. 
 


\section{Recurrence and algebraic actions on measure spaces}\label{rec}


In this section, let $G=\PSL_2(\c)$ and let $\G<G$ be a Zariski
dense geometrically finite discrete subgroup. Set $X=\G\ba G$.
Let $N$ be the horospherical subgroup 
\[
\left\{n_t:=\begin{pmatrix} 1 & 0\\t & 1\end{pmatrix}: t\in \field\right\}
\]
and let $m^{\BR}$
denote the Burger-Roblin measure on $X$ invariant under $N$.

Recall that $m^{\BR}$ is the unique ergodic $N$-invariant Radon measure on $X$ which
is not supported on a closed $N$-orbit.

Let $U<N$ be a non-trivial connected
subgroup of $N$. We denote by
${\mathcal P}(U\ba N)$ the space of probability measures on $U\ba N$.
The natural action $N$ on $U\ba N$ induces an action of $N$ on ${\mathcal P}(U\ba N)$.
 
The aim of this section is to prove the following technical result:
\begin{prop}
\label{lem:Borel-map-var}
If there exists an essentially $N$-equivariant  Borel map
\[
f:(X, m^{\BR})\to {\mathcal P}(U\ba N),
\] 
then $U=N$ and hence $f$ is essentially constant. 
\end{prop}
{For the proof, we will first observe that the $N$ action on ${\mathcal P}(U\ba N)$ is smooth ~\cite[Def.\ 2.1.9]{Zi}.
By the fact that $m^{\BR}$ is $N$-ergodic, it then follows that $f$ is essentially  concentrated on a single $N$-orbit 
 in ${\mathcal P}(U\ba N)$. We will use a recurrence property of $m^{\BR}$, which is stronger than the conservativity, to
  prove $U=N$.

We begin with the following lemma.
The space $\mathcal P(\bbr)$ is equipped with a weak topology: i.e.,
$\nu_n \to \nu$ if and only if $\nu_n(\psi)\to \nu(\psi)$ for all $\psi\in C_c(\br)$.
\begin{lem}\label{lem:prob-R}
 If  $\{t_n:n=1,2, \cdots \}$
is  sequence in $\br$, so that $t_{n*}\sigma\to\sigma'$ for some $\sigma,\sigma'\in{\mathcal P}(\bbr)$, then
 $\{t_{n}\}$ is bounded. 
\end{lem} 

\begin{proof}
Assume the contrary and after passing to a subsequence suppose $t_n\to\infty.$
Since $\sigma$ and $\sigma'$ are probability measures on $\bbr,$
there is some $M>1$ such that 
\[
\sigma([-M,M])>0.9 \quad\text{and}\quad \sigma'([-M,M])>0.9.
\]

Let $\psi \in C_c(\bbr)$ be a continuous function so that $0\leq \psi \leq 1,$
$\psi |_{[-M,M]}=1$ and $\psi |_{(-\infty,-M-1)\cup(M+1,\infty)}=0.$
Since $t_n\to\infty$ we have 
\[
\left( [-M-1,M+1]-t_n\right) \cap[-M-1,M+1]=\emptyset \text{ for all large $n.$}
\]
Therefore, $t_{n*}\sigma (\psi )<0.1$ but $\sigma'(\psi )>0.9,$
which contradicts the assumption that $t_{n*}\sigma\to\sigma'.$  
\end{proof}


As was mentioned above, we will need certain recurrence properties of the action of $N$
on $(X,m^{\BR})$. This will be deduced from
 recurrence properties of the Bowen-Margulis-Sullivan measure
$m^{\BMS}$ on $X$ with respect to the diagonal flow
 ${\rm diag}(e^{t/2},e^{-t/2})$. We normalize
so that $m^{\BMS}$ is the probability measure. These two measures $m^{\BMS}$ and $m^{\BR}$
are quasi-product measures and on weak-stable manifolds (i.e., locally transversal to $N$-orbits),
they are absolutely continuous to each other. 

Set $M=\{{\rm{diag}}(z,z^{-1}): |z|=1\}$. Then $G/M$ can be identified with
the unit tangent bundle of the hyperbolic $3$-space $\bH^3$. 
Hence for every $g\in G$, we can associate a point $g^-$ in the boundary
of $\bH^3$ which is the backward end point of the geodesic determined by the tangent vector $gM$.

Now the set $X_{\rm rad}:=\{\Gamma g\in X: g^- \text{ is a radial limit point of $\Gamma$}\}$
has a full $\BMS$-measure as well as a full $\BR$-measure.
For $x\in X_{\rm rad}$, $n \mapsto xn$ is a bijection $N\to xN$,
and $\mu_x^{\PS}$ denotes the leafwise measure of $m^{\BMS}$, considered as a measure on $N$ (see ~\cite[\S2]{MO}).

We recall the following: 

\begin{thm} [\cite{Rudol}, Theorem 17] \label{thm:Rudol}
For any Borel set $B$ of $X$ and any $\eta>0$, the set 
$$
\left\{x\in X_{\rm rad}:\liminf_T \tfrac{1}{\mu_x^{\PS}(B_N(T))}\int_{B_N(T)}
\chi_{B} (x{n_t})d\mu_x^{\PS}(\tbf)\geq (1-\eta)m^{\BMS}( B )\right\}
$$
has full $\BMS$ measure $m^{\BMS}$.
\end{thm}

\begin{lem}\label{lem:recurrence}  Let $U$ be a one-dimensional connected subgroup of $N$. Then
for every subset $B\subset X$ of positive $\BMS$-measure,
 the set
\[
 \{{n}\in U\ba N: x{n}\in B\}
\]
is unbounded for $m^{\BMS}$-a.e.\ $x\in X.$
\end{lem}

\begin{proof}


We denote by $\rm{Nbd}_R(U)$ the $R$-neighborhood of $U$, i.e.,
$${\rm{Nbd}}_R(U)=\{n_t\in N: |t-s|<R\text{ for some $n_s\in U$}\}.$$
We set $B_N(R):= {\rm{Nbd}}_R(\{e\})$ which is the $R$-neighborhood of $e$.

Let $B\subset X$ be any Borel set of positive BMS-measure. 
Then by Theorem~\ref{thm:Rudol}, there is a $\BMS$ full measure set $X'$ of $X_{\rm rad}$
with the following property:
for all $x\in X'$, there is $T_x>0$ such that if $T>T_x$, then
\be\label{eq:max-erg}
\mu_x^{\PS}\{{n_t}\in B_N(T): xn_t \in B\}\geq 
0.9\, \mu_x^{\PS}(B_N(T))m^{\BMS}(B).
\ee

Let $x\in X'$. Since $x$ is a radial limit point for $\Gamma$, there exists a sequence $T_i\to\infty$
 so that $xa_{-\log T_i}$ converges to some  $y\in\supp(m^{\BMS})$.
Therefore, we have
\be\label{eq:PS-cont}
\mu_{xa_{-\log T_i}}^{\PS}\to\mu_y^{\PS},
\ee
in the space of regular Borel measures on $N$ endowed with  the weak-topology (see~\cite[Lemma 2.1]{MO}).

Moreover, by ~\cite[Lemma 4.3]{MO}, for every $\ep>0,$
there exists  $\rho_0>0$ such that for every $0<\rho\leq\rho_0$ we have
\be\label{eq:var-null}
\mu_y^{\PS}(B_N(1)\cap {\rm{Nbd}}_{\rho}(U))\leq \ep\cdot \mu_y^{\PS}(B_N(1))
\ee
Since 
\[
\tfrac{\mu_x^{\PS}(B_N(T_i)\cap {\rm{Nbd}}_{R}(U)) }{\mu_x^{\PS}(B_N(T_i))}
=\tfrac{\mu_{xa_{-\log T_i}}^{\PS}(B_N(1)\cap {\rm{Nbd}}_{R/T_i}(U)) }{\mu_{xa_{-\log T_i}}^{\PS}(B_N(1))},
\]
it follows from \eqref{eq:PS-cont} and \eqref {eq:var-null} that 
 for every $\ep>0$ and for all sufficiently large $i$ such that $R/T_i<\rho$,
\be\label{eq:nbhd-U}
\mu_x^{\PS}(B_N(T_i)\cap {\rm{Nbd}}_{R}(U))\leq \ep \cdot \mu_x^{\PS}(B_N(T_i)).
\ee

Put $\ep=1/10 \cdot m^{\BMS}(B).$
Given any $j$, there exists $i_j>\max\{j, T_x\}$ such that $$ \mu_x^{\PS}(B_N(T_j))\leq\ep\cdot \mu_x^{U}(B_N(T_{i_j})) .$$
Then for all sufficiently large $i>i_j$, we have
$$ \mu_x^{\PS} (B_N(T_j) ) +\mu_x^{\PS} (B_N(T_i)\cap {\rm{Nbd}}_{R}(U)) \le 2 \ep \mu_x^{\PS} (B_N(T_i)) .$$
Therefore
it follows from  ~\eqref{eq:max-erg}
that for any $j$ and for all $i>i_j$,
\begin{multline*}
 \mu_x^{\PS}\{{n_t}\in B_N(T_i)\setminus (B_N(T_j)\cup  {\rm{Nbd}}_{R}(U) ): x{n_t} \in B\}\geq \\
0.5\mu_x^{\PS}(B_N(T_i))m^{\BMS}(B)>0.
\end{multline*}
This implies that the set of $x$ with $xn_t\in B$ cannot be contained in any bounded neighborhood of $U$, proving the claim.
\end{proof}

\medskip

\noindent
{\it Proof of Proposition ~\ref{lem:Borel-map-var}.} If $U=N$, the claim is clear. Hence we suppose $U$ is a one-dimensional
connected subgroup of $N$.
First by modifying $f$ on a $\BR$-null set, we may assume that 
for all $x\in X$, and for all $n\in N$,
\[
f(x{n})=n_* f(x) .
\]

 Fix a compact subset $Q\subset X_{\rm{rad}}$
such that $f$ is continuous on $Q$ and $m^{\BR}(Q)>0$. This is possible by Lusin's theorem. We claim that
 for some $y\in Q$, the set
\[
\{n\in N: yn\in Q\}
\] 
is unbounded in the quotient space $U\ba N$.

First note that
there exists $\rho_0>0$ such that $QB_N({\rho_0})$ has a positive $\BMS$-measure.

By Lemma~\ref{lem:recurrence} there is a $\BMS$-full measure set $X'$  
so that for all $x\in X'$,
\[
\{{n}\in N : x{n} \in QB_N({\rho_0})\}
\text{
is unbounded in $U\ba N.$ }\] 
Using the fact that $N$ is abelian, the above implies that
\be\label{unb}
\{{n}\in N : y{n} \in Q\}\text{ is unbounded in $U\ba N$ for all $y\in X'N$.}  
\ee

The set $X'N$ is a $\BR$-conull set and $m^{\BR}(Q)>0$.
Therefore, there is some $y\in Q$ which satisfies \eqref{unb}, proving  the claim.
Now, there is a sequence $\{n_{t_i}\in N\}$ such that $n_{t_i}\to \infty$ in $U\ba N$
and that $yn_{t_i}\in Q$ and $yn_{t_i}\to z\in Q.$ The function $f$ is continuous on $Q$. Therefore we get 
\[
(n_{t_i})_*f(y)\to f(z).
\] 
Since $f(y)$ and $ f(z)$ are probability measures on $U\ba N\simeq \br$, and $n_{t_i}\to \infty$ in $U\ba N\simeq \br$,
this contradicts Lemma~\ref{lem:prob-R}. 
This yields that $U=N$ is the only possibility and finishes the proof.
\qed

\section{Proof of Theorems \ref{ergg} and \ref{main}}
We continue the notations set up in the introduction.
Let $\field=\br$ or $\c$ and $G=\gf$. 
Let $\Gamma_1<G$ be a cocompact lattice
and $\Gamma_2<G$ be a geometrically finite and Zariski dense subgroup.
Set $X_i=\Gamma_i\ba G$ for $i=1,2$.
Let $Z=X_1\times X_2$. Let $N<G$ be a horospherical subgroup.
Without loss of generality, we may assume
  $$N:=\left\{n_t:=\begin{pmatrix} 1 & 0\\t & 1\end{pmatrix}: t\in \field\right\}.$$

We denote by $\brg$ the $N$-invariant Burger-Roblin measure on $X_2$; this is unique
up to a constant multiple.


Let $\mu$ be a
$\Delta (N)$-invariant, ergodic, conservative {\it infinite} Radon measure on $Z$.
Let 
$$
\pi: Z\to X_2
$$ 
be the canonical projection. Since $X_1$ is compact,
the push-forward $\pi_*\mu$ defines an $N$-invariant ergodic conservative {\it infinite}
 Radon measure on $X_2$.

\begin{thm} Up to a constant multiple,
 $$\pi_*\mu= \brg\quad\text{or }\quad \pi_*\mu=\operatorname{d}\!n$$ for the $N$-invariant measure $\operatorname{d}\!n$ on 
 a closed orbit $x_2N$  homeomorphic to $\br\times \mathbb S^1$. The latter happens only when $\field=\c$ and
 $\G$ has a parabolic limit point of rank one. \end{thm}
 
 \begin{proof}
 Since $\G_2$ is assumed to be geometrically finite and Zariski dense,
 up to a proportionality, the measure $\pi_*\mu$ is either  $\brg$ or
it is the $N$-invariant measure supported on a closed $N$-orbit $x_2N$ in $X_2 $ (\cite{Roblin} and~\cite{Wi}).
In the latter case, $x_2N$ is homeomorphic to one of the following:
 $\mathbb S^1\times \mathbb S^1$, $\br\times \br$, and $\br\times \mathbb S^1$.
The first possibility cannot happen as that would mean that $\mu$ is a finite measure.
The second possibility would contradict the assumption that $\mu$ is $N$-conservative. 
Hence $x_2N$ must be $ \br\times \mathbb S^1$, up to a homeomorphism.
 \end{proof}
 
 The following is one of the main ingredients of our proof of Theorem \ref{main}, established in \cite{MO}.
\begin{thm} \label{pro} 
One of the following holds, up to a constant multiple:
\begin{enumerate}
 \item $\pi_*\mu= \brg$ and 
 $\mu$ is invariant under $U\times \{e\}$ for a non-trivial connected subgroup  $U$ of $N$;
 \item $\pi_*\mu=\brg$ and the fibers 
 of the map $\pi$ are finite with the same cardinality almost surely. 
  Moreover, in this case, $\mu$ is the graph of the $\BR$-measure in the sense of 
  Theorem \ref{main}(2);
\item $\field=\c$ and $\pi_*\mu=\operatorname{d}\!n$ for the $N$-invariant measure $\operatorname{d}\!n$ on 
 a closed orbit $x_2N$  homeomorphic to $\br\times \mathbb S^1$.
\end{enumerate} 
\end{thm}

\begin{proof}  
For the case when $\pi_*\mu=\brg$,
it follows from~\cite[Thm.~7.12 and Thm.~7.17]{MO} 
either that  the fibers of the map $\pi$ are finite 
with the same cardinality almost surely or that
$\mu$ is invariant under a non-trivial connected subgroup of $N$, yielding the cases (1) and (2).
Indeed~\cite[Thm.\ 7.12]{MO} states this under the assumption that
$\mu$ is an $N$-joining, but all that is used in the proof is the fact that 
the projection of the measure onto one of the factors is the $\BR$ measure.

\end{proof}




\subsection{Proof of Theorem \ref{main}}
\subsubsection{The case of $G=\PSL_2(\br)$.}\label{sec:sl2r}
In this case, $\brg$ is the unique infinite conservative $N$-invariant measure on $X_2$. Therefore
we may assume, after the normalization of $\brg $ if necessary,  that
$\pi_*\mu=\brg$. 
By the standard disintegration theorem, we have
\[
\mu=\int_{X_2} \mu_x \operatorname{d}\!\brg (x)
\]
where $\mu_x$ is a probability measure on $X_1$ for $\brg$-a.e.\ $x$.

Suppose that  Theorem \ref{pro}(1) holds, i.e., $\mu$  is invariant under $N\times \{e\}$.
Then, since every element in the $\sigma$-algebra
\[
\{X_1\times B: B\subset X_2\text{ is a Borel set}\}
\] 
is invariant under $N\times \{e\},$ 
we get that $\mu_x$ is an $N$-invariant probability measure on $X_1$ for $\brg$-a.e.\ $x$.

By the unique ergodicity of $N$
on the compact space $X_1$ \cite{Fu}, we have
\be\label{fin} \mu_x=m^{\Haar}\quad \text{ for $\brg$-a.e.\ $x$;}\ee
hence $\mu=m^{\Haar}\times \brg$.

If Theorem \ref{pro}(2) holds, we obtain that $\mu$ is the graph of the $\BR$-measure
as desired in Theorem \ref{main}.

\subsubsection{The case of $G=\PSL_2(\c)$}

In analyzing the three cases in Theorem \ref{pro},
we use the following  special case of
 Ratner's measure classification theorem \cite{Ra}:
  \begin{thm}\label{ratner} 
  Let  $\Gamma_1<G=\PSL_2(\c)$ be a cocompact lattice. Let $U$ be a one-parameter unipotent
  subgroup of $G$.  Let $L\simeq \PSL_2(\br)$ be the connected subgroup generated by $U$ and its transpose $U^t$. 
  Then  any ergodic $U$-invariant probability measure on $\G_1\ba G$ is
either the Haar measure or a $v^{-1} Lv$-invariant measure supported on a compact orbit 
$\Gamma_1\ba \Gamma_1 g Lv$ 
for some $g\in G$ and $v\in N$. 
\end{thm}
Indeed, the same conclusion
holds for any ergodic $u$-invariant probability measure on $\G_1\ba G$
 for any non-trivial element $u\in U$,
as was obtained in \cite{Shah}.

Also note that in the second case of Theorem~\ref{ratner} 
 the support of the measure is contained in
$yLN$  for some compact orbit $yL$.


We now investigate each case of Theorem \ref{pro} as follows:

\begin{thm} \label{prod} 
For $k=1,2,3$, Theorem~\ref{pro}(k) implies Theorem~\ref{main}(k).
\end{thm}

\begin{proof} Observe first that the case of $k=2$ follows directly from Theorem~\ref{pro}.

Consider the case $k=1$: suppose that $\mu$ is invariant under a subgroup $U\times \{e\}$ for a non-trivial
connected subgroup $U$ of $N$. We normalize $\brg$ so that $\pi_*\mu=\brg$.
It follows from the standard disintegration theorem that
\be\label{eq:mu-disint}
\mu=\int_{X_2} \mu_x \operatorname{d}\!\brg (x).
\ee
{Arguing as in \S\ref{sec:sl2r}, since $\mu$
is invariant under $U\times \{e\},$ we get that  
$\mu_x$ is a $U$-invariant probability measure on $X_1$ for $\brg$-a.e.\ $x$.}
We claim that
\be\label{fin} 
\mu_x=m^{\Haar}\quad \text{ for $\brg$-a.e.\ $x$;}
\ee
this implies $\mu=m^{\Haar}\times \brg$ and finishes the proof in this case. 

We apply Theorem \ref{ratner} to $U$.
Let $L\simeq \PSL_2(\br)$ be defined as in Theorem \ref{ratner}.
Compactness of $\G_1\ba \G_1 g L$ implies that
$g^{-1}\Gamma_1 g\cap L$ is a cocompact lattice of $L$. 
In particular, $g^{-1}\Gamma_1 g\cap L$ is finitely generated and Zariski dense in $L$. 
This implies there are only countably many compact $L$ orbits in $X_1$.

 Let
$\{y_iL : i=0, 1, 2, \ldots \}$ be  the collection of all compact $L$-orbits in $X_1.$ 
Then for $\brg$-a.e.\ $x\in X_2$, we have
\be\label{eq:erg-dec-ratner}
\mu_x=c_x m^{\Haar}+ \sum_i  \mu_{x, i} 
\ee
where $c_x \ge 0$ and $\mu_{x, i}$ is a $U$-invariant finite measure supported in $y_i L N$. 

The set $\{(x_1,x_2): c_{x_2}>0\}$ is a $\Delta(N)$-invariant Borel measurable set. 
Therefore,~\eqref{fin} follows if this set has positive measure.

In view of this, we assume from now that $c_x=0$ for $\brg$-a.e.\ $x$. 
Then
the support of $\mu$ is contained in a countable union 
\[ 
\bigcup_i (y_i L N\times X_2).
\]
Hence for some $i$, 
\be \label{un} 
\mu(y_i LN\times X_2)>0 .
\ee
Without loss of generality, we may assume $i=0$.

Since $y_0 LN\times X_2$ is $\Delta(N)$-invariant and 
$\mu$ is $\Delta(N)$-ergodic,~\eqref{un} implies that $y_0 LN\times X_2$
is $\mu$-conull. Therefore, $\mu_x$ is supported on $y_0LN$ for $\brg$-a.e.\ $x\in X_2$.

For each $n\in N$, let $\eta_n$ be the probability measure supported on $y_0Ln$, invariant under
$n^{-1}Ln$.   Noting that $y_0Ln=y_0Ln'$ if $n\in Un'$,
the map $n\mapsto \eta_n$ factors through $U\ba N$. We also have
\be\label{eta0}  n_0\eta_n=\eta_{nn_0}\quad\text{for any $n, n_0\in N$}.\ee
By Theorem \ref{ratner}, the collection $\{\eta_n: n\in U\ba N\}$
 provides all $U$-invariant ergodic probability measures
on $X_1$ whose supports are contained in $y_0LN$.

Hence the $U$-ergodic decomposition of $\mu_x$ gives that  for a.e. $x\in X_2$,
there is a probability measure $\sigma_x$ on $U\ba N$ such that
\[
\mu_x=\int_{U\ba N}\eta_{n} \operatorname{d}\!\sigma_x(n).
\]

Since $\mu$ is $\Delta(N)$-invariant, we have \be\label{eq:N-mu-x}
\mu_{xn_0}=n_0\mu_x\quad\text{ for $\brg$-a.e. $x\in X_2$ and all $n_0\in N.$ }
\ee

Observe that
\be 
\label{u1} \mu_{xn_0}=\int_{U\ba N }\eta_{n} \operatorname{d}\!\sigma_{xn_0}(n),
\ee 
and
that 
\begin{align*}  
n_0\mu_x&=\int_{X_1} n_0 \eta_n\operatorname{d}\!\sigma_x(n)
=\int_{X_1}  \eta_{nn_0}\operatorname{d}\!\sigma_x(n) =\int_{X_1}  \eta_{n}\operatorname{d}(n_0\sigma_x)(n).
\end{align*}
Therefore \eqref{eq:N-mu-x} implies that for $\brg$-a.e. $x\in X_2$ and for a.e. $n_0\in N$,
\be\label{ns} 
n_0\sigma_x=\sigma_{xn_0}.
\ee 

It follows that the Borel map $f:(X_2,\brg)\to \mathcal P(U\ba N)$ defined by 
\[
f(x):=\sigma_x
\]
is essentially $N$-equivariant for the natural action of $N$ on $\mathcal P(U\ba N)$.
 
As $U$ is one dimensional, this yields  a contradiction to Proposition \ref{lem:Borel-map-var} 
and hence completes
the proof of case $k=1$.

We now turn to the proof of the case $k=3$. The argument is similar to the above case.
Let $x_2N$ be a closed orbit as in the statement of Theorem \ref{pro}(3).
We disintegrate $\mu$ as follows:
\be\label{eq:dis-int}
\mu=\int_{x_2N} \mu_x \operatorname{d}\!n
\ee
where $\mu_x$ is a probability measure on $X_1$ for a.e.\ $x\in x_2N$.
As $x_2N$ is homeomorphic to $\br \times \mathbb S^1$, the stabilizer of $x_2$ in $N$ is  generated by
a unipotent element, say, $u$.
Note that $u$ acts trivially on $x_2N$ and $\Delta(u)$ leaves $\mu$ invariant.
Hence again we have 
\be\label{eq:mux-inv}
\mbox{$\mu_x$ is $u$-invariant almost surely. }
\ee

We apply \eqref{eq:erg-dec-ratner} for $u$-invariant measures $\mu_x$. Let $L\simeq \PSL_2(\br)$ 
denoted the connected closed subgroup
containing $u$ and $u^t$ and let $\{y_iL:i=0,1,\ldots \}$ be the collection of all compact $L$-orbits.
Then for almost every $x \in x_2N$ we write
\[
\mu_x=c_x m^{\Haar}+ \sum_i  \mu_{x, i},
\]
where $\mu_{x, i}$ is a $u$-invariant finite measure supported in $y_i L N$.
As before, if $c_x>0$ on a positive measure subset of $x_2N$, 
then $c_x=1$ almost surely by the $\dN$ ergodicity of $\mu.$
Then $\mu=m^{\Haar}\times dn$; note that this measure is $\dN$
ergodic since $m^{\Haar}$ is $u$-ergodic.   This is the case of Theorem~\ref{main}(3)(a).

Lastly we consider the case when $c_x=0$ almost surely. As before,  
\[
\mu(y_i LN\times x_2N)>0
\] 
for some $i$, and hence almost all $\mu_x$ is supported on one $y_iLN$ by the ergodicity
of $\mu$. We assume $i=0$ without loss of generality. 

Set $U=L\cap N$.
Then $\{\eta_n: n\in U\ba N\}$ (with $\eta_n$ defined as in the previous case)
is the set of all $u$-ergodic probability measures on $X_1$ whose
supports are contained in $y_0LN$  by Theorem \ref{ratner}
and the remark following it. 
Therefore, we get a probability measure $\sigma_x\in \mathcal P(U\ba N)$ such that
\[
\mu_x=\int_{n\in U\ba N} \eta_{n} \operatorname{d}\!\sigma_x(n).
\]
Moreover, $n_*\sigma_x=\sigma_{xn}$ for a.e.\ $x$ and all $n\in N$.

Put $\sigma:=\sigma_x$ for some fixed $x$. Without loss of generality, we assume $x=x_2$.
Then for $\psi\in C_c(Z)$,
$$\mu(\psi) =\int_{n\in U\ba N} \int_{x_2n_0 \in x_2N} \int_Y \psi(yn_0n, x_2n_0) dy \operatorname{d}\!n_0\operatorname{d}\!\sigma(n).$$ 
 
 However for each $n\in U\ba N$,
  $\psi \mapsto\int_{x_2n_0 \in x_2N} \int_Y \psi(yn_0n, x_2n_0) dy \operatorname{d}\!n_0$ defines  a $\Delta(N)$-invariant
measure, and hence by the $\Delta(N)$-ergodicity assumption on $\mu$, $\sigma$ must be a delta measure at a point, say $n\in U\ba N$.
Therefore  we arrive at Theorem \ref{main}(3)(b). 
\end{proof}





\subsection{Proof of Theorem \ref{ergg}}\label{dedd}
 Suppose that  the product measure 
 \[
 \mu:=m^{\Haar}\times \brg
 \] 
 is not ergodic for the action of $\dN$. Let $\Omega$ be the support of $\mu$. We consider the decomposition $\Omega=\Omega_d\cup \Omega_c$
where $\Omega_d$ and $\Omega_c$ are maximal $\dN$-invariant dissipative and conservative
subsets respectively. That is, for any positive measure $S\subset \Omega_d$ (resp.\ $S\subset \Omega_c$),
the Haar measure of $\{n\in N: xn\in S\}$ is finite (resp.\ infinite) for almost all $s\in S$ (see~\cite{Kr}).

Consider the ergodic decomposition of $\mu$.
By Theorem \ref{main}, any ergodic conservative component
in the ergodic decomposition of $\mu$ should be
one of the measures as described in Theorem \ref{main}(2) and \ref{main}(3). 

Now $\mu $ gives measure zero to sets of the form 
\[
(x_1,x_2)\Delta(G)(N\times\{e\})
\]
where $(x_1,x_2)\Delta(G)$ is a closed orbit. 
Moreover, there are only countably many closed $\Delta(G)$ orbits in $Z$.

Also, any closed $N$ orbit $x_2N$ gives rise to the family $x_2NA$ of closed $N$-orbits  where $A$ is
the diagonal subgroup. There are only finitely many such $AN$-orbits in $X_2$, as $\Gamma_2$ is geometrically finite
and hence there are only finitely many $\Gamma$ orbits of parabolic limit points. Therefore $\brg$
gives zero measure to the set of all closed $N$-orbits in $X_2$.
 
It follows that  $\Omega_c$ is trivial and hence the product measure $m^{\Haar}\times \brg$ is 
completely dissipative. This is a contradiction since $X_1$ is compact and $\brg$ is $N$-conservative.
This proves Theorem \ref{ergg}.
\qed

\end{document}